\documentclass[12pt, reqno]{amsart}
\usepackage{amsmath, amsthm, amscd, amsfonts, amssymb, graphicx, color}
\usepackage{setspace}
\usepackage{multicol}
\usepackage[bookmarksnumbered, colorlinks, plainpages]{hyperref}
\hypersetup{colorlinks=true,linkcolor=red, anchorcolor=green, citecolor=cyan, urlcolor=red, filecolor=magenta, pdftoolbar=true}

\newtheorem{theorem}{Theorem}[section]
\newtheorem{lemma}[theorem]{Lemma}

\theoremstyle{definition}

\theoremstyle{remark}

\numberwithin{equation}{section}

\newcommand{\NN}{\mathbb{N}}

\newcommand{\CC}{\mathbb {C}}

\newcommand{\R}{\mathbb{R}}

\newcommand{\D}{\mathbb{D}}
\begin{document}
\setcounter{page}{1}
\title[Differential operator on generalized Fock spaces ]{ A note on the differential operator on generalized Fock spaces }
\author [Tesfa  Mengestie]{Tesfa  Mengestie }
\address{Department of Mathematical Sciences \\
Western Norway University of Applied Sciences\\
Klingenbergvegen 8, N-5414 Stord, Norway}
\email{Tesfa.Mengestie@hvl.no}
\thanks{The author  is partially supported by HSH grant 1244/ H15.}
\begin{abstract}
It has long been known that the differential operator $D$  represents a typical example of unbounded operators on many  Banach spaces including
the classical  Fock spaces, the Fock--Sobolev spaces, and  the generalized Fock spaces where  the weight  decays faster than the  Gaussian weight. In this note we identify  Fock  type  spaces where   the operator  admits boundedness, compactness and membership in the   Schatten $\mathcal{S}_p$ class spectral structures. We also showed that its nontrivial spectrum  while  acting on such spaces is  precisely  the closed unit disk  $\overline{\D}$ in the complex plane.

\end{abstract}

\maketitle

\section{Introduction} \label{1}
 Various  order   differential operators play fundamental  rolls   in many part of mathematics including in the study of differential equations. Nevertheless, the  operator   $Df= f'$  often appears  as  a canonical example of unbounded operators  in many Banach spaces including the very  classical Hilbert space $L^2(\R)$,  the space of continuous functions $ C([a, b])$ with the supremum norm,  and the likes. Its unboundedness on    Fock spaces with  the classical Gaussian weight $ e^{-|z|^2}$ and on generalized Fock  spaces where the weight decays faster than the Gaussian weight  was recently verified in \cite{TM3}. The same  conclusion was also  drawn in \cite{TM5} on the Fock--Sobolev spaces which are typical  examples of  generalized  Fock spaces with  weight decaying  slower than the Gaussian weight. A natural question  to consider is   whether there  could exist spaces of Fock type where this operator admits   richer operator-theoretic properties. Said differently, we would like to know how the function-theoretic properties of  the weight functions generating the spaces  are related to the operator-theoretic properties of $D$. The central  aim of this note  is to investigate  this  and identify  Fock type spaces where the operator $ D$ admits some basic spectral properties.

In view of the  above discussion, if there could  exist generalized  Fock  spaces  on which the operator $D$ acts in a
    bounded fashion, then the associated weight must decay slower than the $k^{th}$ order  Fock--Sobolev spaces with  weight $e^{-\frac{|z|^2}{2}+k\log(1+|z|)},$ where $k$ is a nonnegative integer; see \cite{RCKZ, TM4, TM5} for further information on these spaces. Keeping this in mind, we consider the following setting.

      Let   $m>0$,  $0<p<\infty$,  and $\mathcal{F}_{(m,p)}$ be  a class of generalized  Fock   spaces  consisting of all entire functions $f$ for which
\begin{align*}
\|f\|_{(m, p)}^p= \int_{\CC} |f(z)|^p e^{-p|z|^m} dA(z) <\infty,
\end{align*} where   $dA$ denotes the
usual Lebesgue area  measure on $\CC$.   With this, we plan to find conditions on $m$ (equivalently on the growth of $\psi_m(z)= |z|^m$)  under which $D$ admits boundedness, compactness, and other operator-theoretic  structures while acting between  the spaces  $\mathcal{F}_{(m, p)}.$ It turns out that such structures  do happen to exist   only if the inducing weight function $\psi_m$ grows at a rate much slower than  the corresponding weight function in the classical Gaussian case $\psi_2(z)= |z|^2$.  We precise this in our first main result to follow.
\begin{theorem}\label{thm1}
 \begin{enumerate}
\item Let $0<p\leq q<\infty$ and $m>0$. Then  $D: \mathcal{F}_{(m,p)} \to \mathcal{F}_{(m,q)}$ is
\begin{enumerate}
\item bounded if and only if \begin{align*}m\leq   2-\frac{pq}{pq+q-p}.\end{align*}
\item compact if and only if \begin{align*}m< 2-\frac{pq}{pq+ q-p}.\end{align*}
\end{enumerate}
\item Let $0<q<p<\infty$ and $m>0$. Then the following statements are equivalent.
\begin{enumerate}
\item $D: \mathcal{F}_{(m,p)} \to \mathcal{F}_{(m,q)}$ is bounded;
\item $D: \mathcal{F}_{(m,p)} \to \mathcal{F}_{(m,q)}$ is compact;
\item It holds that
 \begin{align*}
 m<  1-2  \Big(\frac{1}{q}-\frac{1}{p}\Big).
\end{align*}
\end{enumerate}
\end{enumerate}
\end{theorem}
The result effectively  identifies the generalized Fock spaces on which the differential operator admits boundedness and compactness operator-theoretic structures. In particular,  when $p= q,$ the operator $D$ enjoys any of the basic spectral structures  on  $\mathcal{F}_{(m, p)}$ only if the  corresponding weight functions $\psi_m$    grow at most  polynomials of degree not exceeding one. If $p<q$, then $\psi_m$ could grow a  bite faster   as
 \begin{align*}
\frac{pq}{pq+ q-p}<1.
 \end{align*}
 On the other hand, if $p>q$, then $\psi_m$    grows  slower than a polynomials of degree  one.

 We note in passing that if we replace both the domain and target spaces by the corresponding growth type spaces $\mathcal{F}_{(m, \infty)}$ which consist of entire functions $f$ for which
\begin{align*}
 \|f\|_{(m, \infty)}= \sup_{z\in \CC} |f(z)|e^{-|z|^m}<\infty,
\end{align*} the same conclusion, $m\leq 1$,  follows which can be also   seen   for example in \cite{Maria, Harutyunyan} as a particular instance.

Our next main result gives a  condition on  the growth of $|z|^m$ under which $D$ belongs to the Schatten $\mathcal{S}_p(\mathcal{F}_{(m,2)})$ class and also identifies its spectrum.
\begin{theorem}\label{thm2}
\begin{enumerate}
\item Let $0< p<\infty$,   $m>0$, and  $D: \mathcal{F}_{(m,2)} \to \mathcal{F}_{(m,2)}$ is compact. Then it   belongs to the Schatten $\mathcal{S}_p(\mathcal{F}_{(m,2)})$ class   for all $p$.
\item Let $1\leq p<\infty$ and $m>0$, and  $D: \mathcal{F}_{(m,p)} \to \mathcal{F}_{(m,p)}$ is bounded, i.e $m\leq 1$.  Then  its spectrum $\sigma(D)= \{0\}$  whenever $m<1$ and   when  $m= 1$;
\begin{align*}
\sigma(D)=\overline{\{\lambda\in \CC : e^{\lambda z }\in \mathcal{F}_{(m,p)}\}}=\overline{\D}.
\end{align*}
\end{enumerate}
\end{theorem}
\section{Preliminaries} In this  section we collect a few  basic  facts which will be used in the proofs of the main results. From \cite{Olivia2}, the Littlewood--Paley type estimate \footnote{The  notation $U(z)\lesssim V(z)$ (or
equivalently $V(z)\gtrsim U(z)$) means that there is a constant
$C$ such that $U(z)\leq CV(z)$ holds for all $z$ in the set of a
question. We write $U(z)\simeq V(z)$ if both $U(z)\lesssim V(z)$
and $V(z)\lesssim U(z)$.}
\begin{align}
\label{paley}
\|f\|_{(m, p)}^p \simeq |f(0)|^p + \int_{\CC} \frac{|f'(z)|^p e^{-p|z|^m}}{ (1+ |z|)^{p(m-1)}} dA(z)
\end{align} holds  for functions $f$  in the space $\mathcal{F}_{(m,p)}$. Such a formula characterizes the spaces in terms of derivatives, and plays a significant roll  specially in the study of integral operators on the spaces.
\begin{lemma}\label{lem2}
Let $\lambda \in \CC$ and $0<p<\infty$. Then  for each entire function $f$ for which $fe^\lambda \in \mathcal{F}_{(m,p)}$, we have
\begin{align}
\label{spectrum}
\int_{\CC} |f(z)e^{\lambda z}|^p e^{-p|z|} dA(z) \lesssim  \int_{\CC} |f'(z)e^{\lambda z}|^p e^{-p|z|}dA(z).
\end{align}
\end{lemma}
This is a key estimation result which helps us obtain our main result on the spectrum of the operator $D$ in Theorem~\ref{thm2}.
\begin{proof}
The proof of  the lemma follows from some ideas  stemmed  in the proof of Proposition~1 in  \cite{Olivia2}.  We argue in the direction of contradiction and assume that \eqref{spectrum} fails to hold. Then, we can find a sequence of  entire functions $(f_n)$  satisfying $f_ne^\lambda \in \mathcal{F}_{(m,p)}$,
\begin{align*}
\int_{\CC} |f_n(z)e^{\lambda z}|^p e^{-p|z|} dA(z)= 1 \  \ \text{and} \ \  \int_{\CC} |f_n'(z)e^{\lambda z}|^p e^{-p|z|} dA(z) <\frac{1}{n}.
\end{align*}
Now, if  $K$ is a compact subset of $\mathcal{\CC},$ the point evaluation estimate for functions in $\mathcal{F}_{(m,p)}$ ( see the analysis in \cite{Olivia})
gives that
\begin{align*}
|f_n'(z)e^{\lambda z}| \lesssim C  \int_{K} |f_n'(z)e^{\lambda z}|^p e^{-p|z|} dA(z) \leq C \frac{1}{n^p}
\end{align*} for some positive constant $C$ that depends only on  $K$. From this it follows that the sequence $f_n'$ converges to zero uniformly on compact subset of $\CC$.  This shows that $f_n$ also converges  to zero uniformly on the compact subsets. We may rewrite
\begin{align}
\label{newargument}
1=\int_{\CC} |f_n(z)e^{\lambda z}|^p e^{-p|z|} dA(z)= \int_{|z|\leq r} |f_n(z)e^{\lambda z}|^p e^{-p|z|} dA(z)\quad \quad \quad \quad \nonumber\\
+ \int_{|z|>r} |f_n(z)e^{\lambda z}|^p e^{-p|z|} dA(z).
\end{align}
Now the first integral on the right-hand side of $\eqref{newargument}$ tends to zero when $n\to \infty$ since $f_n\to 0$ uniformly on $\{z\in \CC: |z|\leq r\}$. On the other hand, the second integral is the tile of a convergent integral and hence tend to zero when $r\to \infty$, and the contradiction follows.
\end{proof}
We denote by $K_{(m,w)}$ the reproducing kernel of the space $\mathcal{F}_{(m,2)}$ at the point $w\in \CC$.   Because of the reproducing property of the kernel and Parseval identity, it  holds  that
\begin{align*}
%\label{kernel}
K_{(m,w)}(z)=\sum_{n=1}^\infty \langle K_{(m,w)}, e_n\rangle e_n(z) \  \text{and}\ \
\| K_{(m,w)}\|_{(m,2)}^2= \sum_{n=1}^\infty |e_n(w)|^2
\end{align*} for any orthonormal basis $(e_n)_{n\in\NN}$ of $\mathcal{F}_{(m, 2)}$.   An immediate consequence of this representation is  that
  \begin{align}
  \label{sidee}
  \frac{\partial}{\partial\overline{w}}K_{(m,w)}(z)=\sum_{n=1}^\infty e_n(z) \overline{e_n'(w)}, \  \text{and}\ \ \Big\|\frac{\partial}{\partial \overline{w}} K_{(m,w)}\Big\|_{(m,2)}^2= \sum_{n=1}^\infty |e_n'(w)|^2.
  \end{align}
 An explicit expression for the reproducing kernel $K_{(w,m)}$ in  the weighted space $\mathcal{F}_{(m,2)}$ is still  unknown apart from the case when $m= 2$.  From \cite{HH}, we already have an estimate for the norm
\begin{align}
\label{normapproximate}
\|K_{(m,w)}\|_{(m,2)}^2 \simeq |w|^{m-2} e^{2|w|^m}.
\end{align} As a consequence of this, we obtain the following useful estimate for our further consideration.
\begin{lemma}\label{lem3} For each $w\in \CC$, we have the asymptotic estimate
\begin{align}
 \ \Big\|\frac{\partial}{\partial \overline{w}} K_{(m,w)}\Big\|_{(m,2)}^2 \simeq  \|K_{(m,w)}\|_{(m,2)}^2 |w|^{2m-2} \simeq |w|^{3m-4} e^{2|w|^m}.
 \end{align}
\end{lemma}
\begin{proof}
  For simplicity, setting $\Psi(r)= r^{\frac{m-2}{2}} e^{r^m}$ and \begin{align*}
  f(z)=\sum_{n= 0}^\infty \frac{z^n}{ \|z^n\|_{(m,2)}},
  \end{align*}  then we have that
\begin{align*}
M_2(r, f)^2 \simeq \int_{-\pi}^\pi |f(re^{i\theta})|^2 d\theta \simeq (\Psi(r))^2.
\end{align*}
  If we show that
\begin{align*}
\limsup_{r \to \infty} \frac{\Psi''(r)\Psi(r)}{(\Psi'(r))^2} <\infty \ \ \text{and}\ \ \Psi'(r) \simeq \Psi(r) r^{m-1},
\end{align*}
 then our conclusion will follow from Lemma~21 of \cite{Olivia} as
\begin{align*}
M_2(r, f')\simeq \Big\|\frac{\partial}{\partial \overline{w}} K_{(m,w)}\Big\|_{(m,2)}\simeq \Psi'(r) \simeq \Psi(r) r^{m-1}.
\end{align*}To this end, we compute
\begin{align*}
\Psi'(r)= \frac{m-2}{2}r^{\frac{m-4}{2}} e^{r^m} + m r^{\frac{3m-4}{2}} e^{r^m} = e^{r^m} r^{\frac{m-2}{2}-1}\bigg(\frac{m-2}{2}+ mr^m \bigg)\nonumber\\
\simeq e^{r^m} r^{\frac{m-2}{2}-1+m} \simeq  \Psi(r) r^{m-1}.
\end{align*}
Furthermore, a   computation shows that
\begin{align*}
(\Psi'(r))^2 \simeq e^{2r^m}  r^{m-2+ 2(m-1)} \ \ \text{and}\ \ \Psi''(r)\simeq e^{r^m}  r^{\frac{m-2}{2}+ 2(m-1)}
\end{align*} from which we have
\begin{align*}
\limsup_{r \to \infty} \frac{\Psi''(r)\Psi(r)}{(\Psi'(r))^2} \simeq \limsup_{r \to \infty}\frac{e^{2r^m} r^{\frac{m-2}{2}+ 2(m-1)}r^{\frac{m-2}{2}}}{e^{2r^m}  r^{m-2+ 2(m-1)}}\simeq 1.
\end{align*}
  \end{proof}
It has been  a fairly    common practice to  test many operator-theoretic properties on the reproducing kernels for the spaces. In the present setting, no explicit expression is known for the kernel function. Thus, for proving  our mains results, we will rather  use another sequence of test function which replaces the role of the reproducing kernel. Such a sequence was first constructed in \cite{Borch} and  has been further used    by several authors  for example \cite{Olivia,JPP,TM3}.  We introduce the sequence of test functions as follows.  We set
\begin{align*}\tau_m(z)=
\begin{cases}   1,  & 0\leq|(m^2-m)z|<1\\
\frac{|z|^{\frac{2-m}{2}}}{ |m^2-m|^{\frac{1}{2}}},\ \  & |(m^2-m)z|\geq 1.
 \end{cases}
 \end{align*}
 For a sufficiently large positive number $R$, there exists a number $\eta(R)$ such that for any  $w\in \CC$ with $|w|> \eta(R)$, there exists an entire function $f_{(w, R)}$ such that
  \begin{enumerate}
  \item
    \begin{align}
  \vspace{-0.3in}
  \label{test00}
      |f_{(w,R)}(z)| e^{-|z|^m}\leq C \min\Bigg\{ 1,\bigg(\frac{\min\{\tau_m(w), \tau_m(z)\}}{|z-w|}\bigg)^{\frac{R^2}{2}}\Bigg\}  \ \ \ \ \ \ \end{align}for all  $ z\in \CC$ and  for some constant $C$ that depends on $|z|^m$ and $R$. In particular when $z \in D(w, R\tau_m(w))$, the estimate becomes
   \begin{align}
  \label{test0}
   |f_{(w,R)}(z)| e^{-|z|^m}\simeq 1.
    \end{align}
  \item $f_{(w, R)}$ belongs to $\mathcal{F}_{(m,p)}$  and its norm is estimated by
\begin{align}
\label{test}
\| f_{(w,R)}\|_{(m, p)}^p \simeq \tau_m^2(w),\ \ \ \  \eta(R) \leq |w|
\end{align} for all  $p$ in the range $0<p<\infty$.\\
Another important ingredient in our subsequent consideration is the pointwise estimate for subharmonic functions $|f|^p$, namely that
\begin{align}
\label{pointwise}
|f(z)|^p e^{-p |z|^m} \lesssim \frac{1}{\sigma^2\tau_m^2(z)} \int_{D(z, \sigma \tau_m(z))} |f(w)|^pe^{-p|w|^m} dA(w)
\end{align} for all finite exponent $p$  and  a small positive number $\sigma$.  The estimate follows from
Lemma~2  of \cite{JPP}.
 \end{enumerate}
  Next, we   recall  the notion of
  covering for the  space $\CC$.  We  denote by $D(w,r)$ the Euclidean disk centered at $w$ and radius $r>0$.  Then,  we record the following useful covering lemma which is essentially from \cite{Olivia,OVL}.
\begin{lemma}\label{lem4}
 Let $\tau_m$ be as above.  Then, there exists a positive $\sigma >0$ and a sequence of points $z_j$ in $\CC$ satisfying the following conditions.
 \vspace{-0.2in}
 \begin{enumerate}
 \begin{multicols}{2}
 \item $z_j\not\in D(z_k,\sigma \tau_m(z_k)), \ \ j \neq k$;
 \item $\CC= \bigcup_jD(z_j, \sigma \tau_m(z_j))$;
  \end{multicols}
 \vspace{-0.2in}
 \item $\bigcup_{z\in  D(z_j, \sigma \tau_m(z_j))}D(z, \sigma \tau_m(z)) \subset D(z_j, 3\sigma \tau_m(z_j))$;
 \item The sequence $ D(z_j, 3\sigma \tau_m(z_j))$ is a covering of $\CC$ with finite multiplicity $N_{\max}$.
 \end{enumerate}
\end{lemma}
 \begin{lemma}\label{lem5}
Let $R$ be a sufficiently large number and  $\eta(R)$ be as before. If $(z_k)$  is the  covering sequence from Lemma~\ref{lem4}, then  the function
\begin{align*}
F= \sum_{z_k: |z_k| > \eta(R)} a_k \frac{f_{(z_k, R)}}{\tau_m^{\frac{2}{p}}(z_k)}
\end{align*} belongs to $\mathcal{F}_{(m, p)}$ for every sequence $(a_k)$ in $\ell^p$ and also $\|F\|_{(m,p)} \lesssim \|(a_k)\|_{\ell^p}$.
 \end{lemma}
 The proof the Lemma follows from a simple variant of the proof of Proposition~9 in \cite{Olivia} or Proposition~1 in \cite{JPP}.
\section{Proof of the Main results}
\subsection{Proof of Theorem~\ref{thm1}-Part (i)}
 Let us first prove the necessity of the condition  in part (i),  and assume  that  $D:\mathcal{F}_{(m,p)} \to \mathcal{F}_{(m,q)}$ is bounded. Then,   making use of the estimates in  \eqref{test0},  \eqref{test}, \eqref{test00} and \eqref{pointwise}, we have
\begin{align*}
 \|D\|^q \gtrsim \tau_m^{-\frac{2q}{p}}(w) \|Df_{(w, R)}\|_{(m,q)}^q= \tau_m^{-\frac{2q}{p}}(w)\int_{\CC}|f'_{(w, R)}(z)|^q e^{-q|z|^m} dA(z)\quad \quad \quad \quad \quad \quad \quad \quad  \nonumber\\
\geq \tau_m^{-\frac{2q}{p}}(w) \int_{D(w, \delta \tau_m(w))} \frac{ |f'_{(w, R)}(z)|^q}{ e^{q|z|^m}} dA(z)\gtrsim \tau_m^{2-\frac{2q}{p}}(w)  \frac{ |f'_{(w, R)}(w)|^q }{ e^{q|w|^m}}\quad \quad \quad \quad \quad \quad \quad \quad \quad \quad \quad \quad \nonumber\\
\simeq m^q\tau_m^{2-\frac{2q}{p}}(w) |w|^{q(m-1)}\quad \quad \quad \quad \quad \quad \quad \quad\quad \quad \quad \quad
\end{align*} for all $w\in \CC$. It follows that
\begin{align}
\label{normD}
 \|D\| \gtrsim \begin{cases}
|m^{2+p}-m^{1+p}|^{\frac{1}{p}}\sup_{w\in \CC} \big(1+|w|\big)^{(m-1)+\frac{(q-p)(m-2)}{qp}}, & m\neq1.\\
 1,  & m=1
 \end{cases}
\end{align} which holds only if $ pq(m-1)+(q-p)(m-2) \leq 0$ as asserted,  and  it also gives  a one sided estimate for the norm of $D$.

  We now turn to the proof of the sufficiency of the condition in part (i).  We use the covering sequences  approach along with Lemma~\ref{lem4},  where the original  idea goes back to \cite{OVL}.  Applying \eqref{paley} and \eqref{pointwise},  we  estimate
 \begin{align}
  \|Df\|_{(m,q)}^q = \int_{\CC} |f'(z)|^qe^{-q|z|^m} dA(z)\leq \sum_{j} \int_{D(z_j, \sigma\tau_m(z_j))} |f'(z)|^q e^{-q|z|^m}dA(z)\nonumber\\
 \lesssim \sum_j  \int_{D(z_j, \sigma\tau_m(z_j))}\bigg( \frac{1}{\tau_m^2(z)} \int_{D(z, \sigma\tau_m(z))} |f'(w)|^pe^{-p|w|^m} dA(w)\bigg)^{\frac{q}{p}} dA(z)=: S \nonumber
 \end{align}
 Now for each point $z\in D(w, \sigma\tau_m(w))$, observe  that  $1+|z| \simeq 1+|w|$. Taking this into account,  we further estimate
 \begin{align}
S \simeq \sum_j  \int_{D(z_j, \sigma\tau_m(z_j))}\bigg( \frac{m^p(1+|z|)^{p(m-1)}}{\tau_m^2(z)} \int_{D(z, \sigma\tau_m(z))} \frac{|f'(w)|^pe^{-p|w|^m}}{m^p(1+|w|)^{p(m-1)}} dA(w)\bigg)^{\frac{q}{p}} dA(z)\nonumber\\
 \lesssim \sum_j  \bigg( \int_{D(z_j, 3\sigma\tau_m(z_j))} \frac{|f'(w)|^pe^{-p|w|^m}}{m^p(1+|w|)^{p(m-1)}} dA(w)\bigg)^{\frac{q}{p}} \int_{D(z_j, \sigma\tau_m(z_j))} \frac{m^q(1+|z|)^{q(m-1)}}{\tau_m^{\frac{2q}{p}}(z)} dA(z)\nonumber
  \end{align}
  Since $q\geq p,$ applying Minkowski inequality and the finite multiplicity $N_{\max}$ of the covering sequence $D(z_j, 3\sigma\tau_m(z_j))$,  we obtain
\begin{align*}
\sum_j  \bigg( \int_{D(z_j, 3\sigma\tau_m(z_j))} \frac{|f'(w)|^pe^{-p|w|^m}}{m^p(1+|w|)^{p(m-1)}} dA(w)\bigg)^{\frac{q}{p}} \int_{D(z_j, \sigma\tau_m(z_j))} \frac{m^q(1+|z|)^{q(m-1)}}{\tau^{\frac{2q}{p}}_m(z)} dA(z)\quad \nonumber\\
\leq \bigg( \sum_j \int_{D(z_j, 3\sigma\tau_m(z_j))} \frac{|f'(w)|^pe^{-p|w|^m}}{m^p(1+|w|)^{p(m-1)}} dA(w)\bigg)^{\frac{q}{p}} \int_{D(z_j, \sigma\tau_m(z_j))} \frac{m^q(1+|z|)^{q(m-1)}}{\tau_m^{\frac{2q}{p}}(z)} dA(z)\nonumber\\
 \lesssim \|f\|_{(m,p)}^q \sup_{w\in \CC} \int_{D(w, \sigma\tau_m(w))} \frac{m^q(1+|z|)^{q(m-1)}}{\tau_m^{\frac{2q}{p}}(z)} dA(z)\quad \quad \quad \quad \quad \quad \quad \quad\nonumber\\
 \lesssim  \|f\|_{(m,p)}^q  \sup_{w\in \CC} \frac{m^q(1+|w|)^{q(m-1)} \tau_m^2(w)}{\tau_m^{\frac{2q}{p}}(w)}\quad \quad \quad \quad \quad \quad \quad \quad\nonumber\\
 \simeq \|f\|_{(m,p)}^q |m^{2+p}-m^{1+p}|^{\frac{q}{p}} \sup_{w\in \CC} (1+|w|)^{q(m-1)+\frac{q-p}{p}(m-2)} \quad \quad \quad \quad \quad \quad \quad \quad
 %\label{series4}
 \end{align*} from which the sufficiency of the condition and
the reverse side of the estimate in \eqref{normD} follow. Thus we estimate the norm by
\begin{align*}
 \|D\| \simeq \begin{cases}
|m^{2+p}-m^{1+p}|^{\frac{1}{p}}\sup_{w\in \CC} \big(1+|w|\big)^{(m-1)+\frac{(q-p)(m-2)}{qp}}, & m\neq1.\\
 1,  & m=1
 \end{cases}
\end{align*}
To prove the compactness, we first assume that  the condition $m<2-\frac{pq}{pq+q-p}$ holds. Then for  each positive $\epsilon$,  there exists $N_1$ such that
 \begin{align}
|m^{2+p}-m^{1+p}|^{\frac{1}{p}} \sup_{|w|> N_1}
 (1+|w|)^{q(m-1)+ \frac{(q-p)(m-2)}{qp} } <\epsilon.
 \label{partly}
 \end{align}
 Next, we let   $f_n$ to  be a uniformly bounded sequence of functions in $\mathcal{F}_{(m,p)}$ that converges uniformly to zero on   compact subsets of  $\CC$.  Then applying   \eqref{paley} and  arguing in the same way as in the  series of estimations made  above, and invoking eventually \eqref{partly} it follows that
  \begin{align*}
    \|Df_n\|_{(m,q)}^q \lesssim  \int_{|z|\leq N_1}\frac{|f_n'(z)|^q}{ e^{q|z|^m}} dA(z)
    + \sum_{|z_j|>N-1}\int_{D(z_j,\sigma \tau_m(z_j))} \frac{|f_n'(z)|^q}{ e^{q|z|^m}} dA(z)\quad \quad \quad \quad \quad \quad \quad \quad \quad \quad \quad \quad\quad \quad \quad \quad\nonumber\\
       \lesssim \sup_{|w|\leq N_1} |f_n(w)|^q + \quad \quad \quad \quad \quad \quad \quad \quad \quad \quad \quad \quad \quad \quad \quad \quad \quad \quad \quad\quad \quad \quad \quad\quad \quad \quad \quad\nonumber\\
     + \sum_{|z_j|>N-1}\int_{D(z_j,\sigma \tau_m(z_j))} \bigg( \frac{m^p(1+|z|)^{p(m-1)}}{\tau_m^2(z)} \int_{D(z, \sigma\tau_m(z))} \frac{|f_n'(w)|^pe^{-p|w|^m}}{m^p(1+|w|)^{p(m-1)}} dA(w) \bigg)^{\frac{q}{p}} dA(z) \quad \quad \quad \quad \quad \quad \quad \quad \quad \quad \quad \quad\quad \quad \quad \quad\nonumber\\
   \lesssim \sup_{|w|\leq N_1} |f_n(w)|^q  + \|f_n\|_{(m,q)}^q \sup_{|w|>N_1} \frac{m^q(1+|w|)^{q(m-1)} \tau_m^2(w)}{\tau_m^{\frac{2q}{p}}(w)}\quad \quad \quad \quad \quad \quad \quad \quad \quad \quad \quad \quad \quad \quad \quad \quad \quad\nonumber\\
      \lesssim \sup_{|w|\leq N_1} |f_n(w)|^q+ |m^{2+p}-m^{1+p}|^{\frac{1}{p}} \sup_{|w|>N_1} (1+|w|)^{q(m-1)+ \frac{q-p}{p}(m-2)}
   \lesssim \epsilon \  \ \text{as}\ \ n\to \infty.\quad \quad \quad \quad \quad \quad \quad \quad \quad \quad \quad \quad \quad \quad \quad \quad
 \end{align*}
  Conversely, assume that $D$ is compact, and observe that the  normalized  sequence  $ f^*_{(w,R)}= f_{(w,R)}/\|f_{(w,R)}\|_{(m,p)}$ converges to zero as $|w|  \to \infty,$ and the convergence is uniform on compact subset of $\CC.$ Then  applying \eqref{pointwise} and \eqref{test0}, we find
 \begin{align*}
 \frac{|w|^{q(m-1)}}{\tau_m^{2\frac{q-p}{p}}(w)}\simeq
  (1+|w|)^{q(m-1)}\tau_m^{2\frac{p-q}{p}+\frac{2q}{p}}(w)e^{-q|w|^m} |f^*_{(w, \eta(R))}(w)|^q\nonumber\\
  \lesssim  \int_{D(w, \sigma\tau_m(w))} (1+|z|)^{q(m-1)} |f^*_{(w, \eta(R))}(z)|^qe^{-q|z|^m} dA(z)\nonumber\\
\lesssim   \|Df^*_{(w, \eta(R))}\|_{(m,q)}^q  \to 0, \ \ \text{as} \ \ |w| \to \infty.
  %\label{new}
   \end{align*}
  We note in passing that in particular when $p= q$ the necessary of the conditions in part (i) could  be also established   using the sequence of the polynomials $(z^n)$ as test functions. Such polynomials belong to the spaces $\mathcal{F}_{(m,p)}$ for all $p$.  Because arguing with polar coordinates and subsequently substitution,  we could easily observe that
   \begin{align*}
  \|z^n\|_{(m, p)}^p= \int_{\CC} |z^n|^p e^{-p|z|^m} dA(z)= 2\pi\int_{0}^\infty r^{pn+1} e^{-pr^m} dr\quad \quad \nonumber\\
= 2\pi  p^{-\frac{pn+2}{m}} \int_{0}^\infty  t^{\frac{pn+2}{m}-1} e^{-t} dt
=  2\pi  p^{-\frac{pn+2}{m}}\Gamma \Big( \frac{pn+2}{m} \Big)<\infty.
\end{align*}

  For the case $p<q$, an application of such  polynomials  only gives the condition
  \begin{align*}
  m\leq 2-\frac{2(pq-3(q-p))}{p-q+ 2pq}
  \end{align*} which is weaker than the condition in the result  since
  \begin{align*}
  \frac{2(pq-3(q-p))}{p-q+ 2pq} < \frac{pq}{pq+q-p},\  \text{for}\ \  p<q.
  \end{align*}
\subsection{Proof of Theorem~\ref{thm1}-Part (ii)} We assume $0<q<p<\infty$.  As  $(b)$ obviously implies (a), we plan to show  (a) implies (c) and (c) implies (b).
 For the first, we follow this classical technique  where the original idea goes back to Luecking \cite{DL}.  Let $0<q<\infty$ and $R$  be a sufficiently large number   and $(z_k)$ be the covering sequence as in Lemma~\ref{lem4}. Then by  Lemma~\ref{lem5},
\begin{align*}
F= \sum_{z_k:|z_k|\geq\eta(R)} a_k  \frac{f_{(z_k,R)}}{\tau_m^{\frac{2}{p}}(z_k)}
\end{align*} belongs to $\mathcal{F}_{(m,p)}$ for every $\ell^p$  sequence $(a_k)$  with norm  estimate
$
%\label{series}
\|F\|_{(m,p)} \lesssim \|(a_k)\|_{\ell^p}.$
If $(r_k(t))_k$  is the Radmecher sequence of function on $[0,1]$ chosen as in \cite{DL}, then the sequence $(a_kr_k(t))$ also
belongs to $\ell^p$ with $\|(a_kr_k(t))\|_{\ell^p}= \|(a_k)\|_{\ell^p}$ for all $t$. This implies that the function
\begin{align*}
F_t= \sum_{z_k:|z_k|\geq\eta(R)} a_k r_k(t) \frac{f_{(z_k,R)}}{\tau_m^{\frac{2}{p}}(z_k)}
\end{align*} belongs to $\mathcal{F}_{(m,p)}$  with norm estimate
$\|F_t\|_{(m,p)} \lesssim \|(a_k)\|_{\ell^p}. $ Then, an  application of Khinchine's inequality \cite{DL} yields
\begin{align}
\label{Khinchine}
\Bigg(\sum_{z_k:|z_k|\geq\eta(R)} |a_k|^{2} \frac{|f_{(z_k,R)}'(z)|^2}{\tau_m^{\frac{4}{p}}(z_k)}\Bigg)^{\frac{q}{2}}\lesssim \int_{0}^1\bigg| \sum_{z_k:|z_k|\geq\eta(R)} a_k r_k(t) \frac{f'_{(z_k,R)}(z)}{\tau_m^{\frac{2}{p}}(z_k)}\bigg|^q dt.
\end{align}
Making use of \eqref{Khinchine}, and subsequently Fubini's theorem, we have
 \begin{align*}
\int_{\CC}\Bigg(\sum_{z_k:|z_k|\geq\eta(R)} |a_k|^{2} \frac{|f_{(z_k,R)}'(z)|^2}{\tau_m^{\frac{4}{p}}(z_k)}\Bigg)^{\frac{q}{2}}e^{-q|z|^m}dA(z)\quad \quad \quad \quad \quad \quad \quad  \quad \quad \quad  \quad \quad \quad  \quad \quad \quad \nonumber\\
\lesssim \int_{\CC} \int_{0}^1\bigg| \sum_{z_k:|z_k|\geq\eta(R)} a_k r_k(t) \frac{f'_{(z_k,R)}(z)}{\tau_m^{\frac{2}{p}}(z_k)}\bigg|^q dt e^{-q|z|^m}dA(z)\nonumber\\
=  \int_{0}^1 \int_{\CC}\bigg| \sum_{z_k:|z_k|\geq\eta(R)} a_k r_k(t) \frac{f'_{(z_k,R)}(z)}{\tau_m^{\frac{2}{p}}(z_k)}\bigg|^q e^{-q|z|^m}dA(z) dt\simeq \int_{0}^1\|D F_t\|_{\mathcal{F}_{(m,q)}}^q dt\lesssim \|(a_k)\|_{\ell^p}^q.
%\label{series3}
\end{align*}
Now arguing with this, the covering lemma,  and \eqref{test0} leads to the series of estimates
\begin{align*}
\sum_{z_k:|z_k|\geq\eta(R)}\frac{ |a_k|^{q}}{\tau_m^{\frac{2q}{p}}(z_k)}\int_{D(z_k,3\sigma\tau_m(z_k))}(1+ |z|)^{q(m-1)} dA(z) \quad \quad \quad \quad \quad \quad \quad \quad \quad \quad \quad \quad  \quad  \nonumber\\
\simeq  \sum_{z_k:|z_k|\geq\eta(R)}\frac{ |a_k|^q}{\tau_m^{\frac{2q}{p}}(z_k)}\int_{D(z_k, 3\sigma\tau_m(z_k))} |f'_{(z_k, R)}(z) |^qe^{-q|z|^m} dA(z)\nonumber\\
 \simeq \int_{\CC} \sum_{z_k:|z_k|\geq\eta(R)}\frac{ |a_k|^{q}}{\tau_m^{\frac{2q}{p}}(z_k)}\chi_{D(z_k, 3\sigma\tau_m(z_k))}(z)| f'_{(z_k, R)}(z)|^q e^{-q|z|^m}dA(z)\nonumber\\
\lesssim \max\{1, N_{\max}^{1-q/2}\} \int_{\CC}\Bigg(\sum_{z_k:|z_k|\geq\eta(R)} |a_k|^{2} \frac{|f_{(z_k,R)}'(z)|^2}{\tau_m^{\frac{4}{p}}(z_k)}\Bigg)^{\frac{q}{2}}e^{-q|z|^m}dA(z)
\lesssim \|(a_k)\|_{\ell^p}^q.
\end{align*}
Applying duality between the spaces $\ell^{p/q}$ and $\ell^{p/(p-q)}$, we again get
\begin{align*}
\sum_{z_k:|z_k|\geq\eta(R)}\Bigg(\frac{1}{\tau_m^2(z_k)}\int_{D(z_k, 3\sigma\tau_m(z_k))}(1+ |z|)^{q(m-1)} dA(z)\Bigg)^{\frac{p}{p-q}}\tau_m^2(z_k)\quad \quad \quad \quad   \nonumber\\
\simeq \sum_{z_k:|z_k|\geq\eta(R)} (1+|z_k|)^{\frac{qp(m-1)}{p-q}} \tau_m^2(z_k) < \infty. \quad \quad \quad \quad \quad \quad \quad  \quad
\end{align*}
On the other hand, we can find a positive number $r\geq \eta(R)$ such that whenever a point $z_k$ of the covering sequence $(z_j)$ belongs to
$\{|z|<\eta(R)\}$, then $D(z_k, \sigma\tau_m(z_k)) $ belongs to $\{|z|<\eta(R)\}$. In view of this we estimate
\begin{align*}
%\label{againn}
\int_{|w|\geq r} (1+|w|)^{\frac{qp(m-1)}{p-q}}dA(w)
\leq \sum_{|z_k|\geq\eta(R)}\int_{D(z_k, \sigma\tau_m(z_k))} (1+|w|)^{\frac{qp(m-1)}{p-q}}dA(w)\quad \quad \quad \nonumber\\
\lesssim \sum_{|z_k|\geq\eta(R)}\int_{D(z_k, \sigma\tau_m(z_k))} (1+|w|)^{\frac{qp(m-1)}{p-q}}\tau^2_m(z_k)dA(w)\nonumber\\
\simeq \sum_{|z_k|\geq\eta(R)}(1+|z_k|)^{\frac{qp(m-1)}{p-q}} \tau^2_m(z_k)  < \infty.
\end{align*}
It also  follows that
\begin{align*}
%\label{Okagain}
\int_{|w|< r} \Bigg(\frac{1}{\tau_m^2(w)}\int_{D(w, 3\delta\tau_m(w))} (1+|z|)^{q(m-1)} dA(z)\Bigg)^{\frac{p}{p-q}}dA(w)< \infty
\end{align*}
Taking into account  the range of   the above estimates we find
\begin{align*}
 \int_{\CC} (1+|z|)^{\frac{qp}{p-q}(m-1)}dA(w)= \int_{|z|\leq r} (1+|w|)^{\frac{qp}{p-q}(m-1)}dA(w)\quad   \nonumber\\
 + \int_{|w|>r} (1+|w|)^{\frac{qp}{p-q}(m-1)}dA(w)<\infty,
\end{align*}  which holds only if
$\frac{qp}{p-q}(m-1) <-2$ as claimed.

To prove (c) implies (b), we argue as follows. Let   $f_n$ to  be a uniformly bounded sequence of functions in $\mathcal{F}_{(m,p)}$ that converges uniformly to zero on   compact subsets of  $\CC$, and by the given condition, for each $\epsilon >0,$ there exists a positive number $r_1$ such that
 \begin{align}
 \label{less}
 \int_{|z|>r_1} (1+|z|)^{\frac{qp}{p-q}(m-1)} dA(z) < \epsilon.
 \end{align}
 Since $p/q >1, $   applying H\"older's inequality, \eqref{paley} and \eqref{less}, we have
  \begin{align*}
  \int_{|z|>r_1}|f_n'(z)|^q e^{-q|z|^m}dA(z)= \int_{|z|>r_1}\bigg( \frac{|f_n'(z)|^qe^{-q|z|^m}}{(1+|z|)^{q(m-1)}}\bigg) (1+|z|)^{q(m-1)}dA(z)\nonumber\\
  \lesssim \|f\|_{(m,p)}^q \bigg(\int_{|z|>r_1} (1+|z|)^{\frac{qp}{p-q}(m-1)}dA(z) \bigg)^{\frac{p-q}{p}}\nonumber\\
  \lesssim \|f\|_{(m,p)}^q \epsilon \lesssim \epsilon. \quad \quad \quad \quad \quad \quad \quad  \quad
 % \label{onepart}
  \end{align*}
  On the other hand when $|z|\leq r_1$, then
  \begin{align*}
  \int_{|z|\leq r_1} |f_n'(z)|^q e^{-q|z|^m}dA(z)\lesssim \int_{|z|\leq r_1} |f_n(z)|^q(1+|z|)^q  e^{-q|z|^m}dA(z)\nonumber\\
   \lesssim \sup_{|z|\leq r_1}|f_n(z)|^q  \int_{|z|\leq r_1} (1+|z|)^q  e^{-q|z|^m}dA(z)\nonumber\\
   \lesssim \sup_{|z|\leq r_1}|f_n(z)|^q  \to 0 \ \ \text{as}\ \  n\to \infty
    \end{align*} from which our claim follows.
\subsection{Proof of Theorem~\ref{thm2}}
\emph{Part (i).}  Let us now turn to the Schatten $\mathcal{S}_p(\mathcal{F}_{(m,2)})$ membership of  $D$.  We recall that a compact  $D$ belongs
to the Schatten $\mathcal{S}_p(\mathcal{F}_{(m,2)})$ class if and only if the sequence of the eigenvalues of  the positive operator $(D^*D)^{1/2}$  is  $\ell^p$ summable.  It suffices to prove the statement for $p\geq 1$. The remaining case for $0<p< 1$ follows by the monotonicity property $\mathcal{S}_p(\mathcal{F}_{(m,2)})\supseteq \mathcal{S}_q(\mathcal{F}_{(m,2)})$ for  $p\leq q$. \\

If  $p>1,$ then   $D$ belongs to $\mathcal{S}_p(\mathcal{F}_{(m,2)})$ if and only if
\begin{align}
\label{sum}
 \sum_{n=0}^\infty |\langle De_n, e_n\rangle|^p <\infty,
 \end{align}
  for any orthonormal basis $(e_n)$ of  $\mathcal{F}_{(m,2)}$ (see \cite[Theorem 1.27]{KZH1}).
Note that the sequence of the  polynomials $(z^n/\|z^n\|_{(m,2)})$  constitutes an orthonormal basis to $\mathcal{F}_{(m,2)}$.  Since
$$De_n= n\frac{z^{n-1}}{\|z^n\|_{(m,2)}}=\frac{n \|z^{n-1}\|_{(m,2)}}{\|z^n\|_{(m,2)}} e_{n-1}, $$
we obtain
\begin{align*}
\langle De_n, e_n\rangle= \frac{n \|z^{n-1}\|_{(m,2)}}{\|z^n\|_{(m,2)}} \langle e_{n-1}, e_n\rangle=0
\end{align*} for all n, from which \eqref{sum} easily follows.\\

\emph{Part (ii).}
   Recall that the spectrum $\sigma (T)$ of   a bounded  operator $T$ is the set containing  all $\lambda \in \CC$ for which
   $\lambda I-T$  fails to be invertible, where $I$ is the identity operator. The complement of the spectrum is referred as    the resolvent set. \\
    A simple computation shows that the  function $f^*(z)= c e^{\lambda z}$ solves the differential equation $\lambda f= Df= f'$, where $c$ is a constant. Then we  analyze
\begin{align}
\label{new}
\|f^*\|_{(m,p)}^p=\int_{\CC} |ce^{\lambda z}|^p e^{-p|z|^m} dA(z)= |c|^p \int_{\CC} e^{p\Re(\lambda z)-p|z|^m} dA(z)
\end{align} depending  on the size of $m$.   Let us first assume that $m=1$.  Then,   the integral in \eqref{new} converges  for each $\lambda \in \CC$  such that  $ |\lambda| <1$.  This means that the function $f^*$ belongs to $ \mathcal{F}_{(m,p)}$, and can be chosen in such a way that $c\neq 0$. From this we deduce
 \begin{align}
\label{reverse}
\overline{\D}\subseteq \sigma(D)\ \ \text{or}\ \ \overline{\{\lambda \in \CC: e^{\lambda z}\in \mathcal{F}_{(m,p)}\}} \subseteq \sigma(D).
\end{align} To prove the  reverse  inclusion  in  \eqref{reverse}, observe that the  integral in \eqref{new} fails to converge  for each $|\lambda |\geq 1$ and $c\neq 0$. It means that  $\lambda I-D$ is injective whenever $|\lambda|\geq1.$  On the other hand, for such values of  $\lambda,$  a  simple computation again shows that the equation $ \lambda f-Df= h$ has a unique analytic  solution
\begin{align}
\label{explicit}
f(z)= R_{\lambda} h(z)= Ce^{\lambda z} - e^{\lambda z} \int_{0}^z  e^{-\lambda w} h(w) dA(w),
\end{align} where $R_\lambda$ is the resolvent operator of $D$ at point $\lambda$, $C= f(0)$ is a constant value.
We remain to show that the operator  $R_{\lambda}$  given by the explicit expression in \eqref{explicit} is  bounded on $\mathcal{F}_{(m, p)}.$  To this end, applying Lemma~\ref{lem2}, we have
\begin{align*}
\|R_{\lambda} h\|_{(m, p)}^p =  \bigg\| e^{\lambda z}\bigg( C- \int_{z_0}^z  e^{-\lambda w} h(w) dA(w)\bigg) \bigg\|_{(m, p)}^p \quad \quad \quad \quad \nonumber\\
\lesssim \int_{\CC} |e^{\lambda z}|^p e^{-p|z|} \bigg|\frac{d}{dz}  \int_{z_0}^z  e^{-\lambda w} h(w) dA(w)\bigg|^p dA(z)\nonumber\\
\lesssim \int_{\CC} |h(z)|^p e^{-p|z|}  dA(w) dA(z) = \|  h \|_{(m, p)}^p. \quad \quad \quad \nonumber
\end{align*}
We now turn to the case  $m<1$. For this,  part  (b) of our result forces $D$ to be a compact operator. Furthermore, we observe that  the integral in \eqref{new} converges only if $c=0$ and hence $f^*(z)= 0$,  which clarifies that  $D$ has no point spectrum. To this effect, $\sigma(D)= \{0\}$.
  

\begin{thebibliography}{BRSHZE}
  \bibitem{Maria} M. J. Beltr\'{a}n, Dynamics of differentiation and integration operators on weighted space of entire functions, Studia Matematica, 221, \textbf{1}(2014), 35--60.
   \bibitem{HH}     H. Bommier-Hatoa, M. Englis,  El-Hassan Youssfia, Bergman-type projections in generalized Fock spaces. J. Math. Anal. Appl., 389,\textbf{ 2} (2012), 1086--1104.
\bibitem{Borch} A. Borichev, R. Dhuez and K. Kellay, Sampling and interpolation in large Bergman and Focks spaces. J. Funct. Anal.,  \textbf{242} (2007), 563--606.
\bibitem{RCKZ} R. Cho  and  K. Zhu,  Fock--Sobolev spaces and their Carleson measures,  J. Funct. Anal.,  Vol. 263, Issue 8, \textbf{15}  (2012), 2483--2506.


    \bibitem{Olivia2} O. Constantin and Ann-Maria  Persson, The spectrum of Volterra type integration operators on generalized  Fock spaces,
 Bull. London Math. Soc. 47(2015), no 6, 958--963.

 \bibitem{Olivia} O. Constantin and Jos\'{e} \'{A}ngel Pel\'{a}ez, Integral Operators, Embedding Theorems and a Littlewood--Paley Formula on Weighted Fock Spaces,  J. Geom. Anal.,  (2015),  1--46.
     \bibitem{Harutyunyan}  A. Harutyunyan, W. Lusky,
On the boundedness of the differentiation operator
between weighted spaces of holomorphic functions, Studia Math. \textbf{184}(2008),233--247.
     \bibitem{DL} D.~Luecking, Embedding theorems for space of  analytic functions via Khinchine's inequality,  Michigan Math. J.,
\textbf{40} (1993),  333--358.


\bibitem{TM3} T. Mengestie and S.  Ueki,  Integral, differential  and multiplication operators on weighted Fock spaces, Preprint, 2016.
    \bibitem{TM4} T. Mengestie, Carleson type measures for Fock--Sobolev spaces,
Complex Anal. Oper. Theory,  \textbf{8} (2014), no 6, 1225--1256.

\bibitem{TM5} T. Mengestie, Spectral properties  of Volterra-type integral operators on Fock--Sobolev  spaces,   Korean Journal of Mathematical Society, https://doi.org/10.4134/JKMS.j160671.

\bibitem{OVL} V. L,  Oleinik,  Embedding theorems for weighted classes of harmonic and analytic functions. J. Math.
Sci.,  \textbf{9(2)} (1978),  228--243.
\bibitem{JPP} J. Pau and J. A. Pel\'{a}ez,  Embedding theorems and  integration operators on Bergman spaces  with rapidly decreasing weightes. J. Funct. Anal.,  259 \textbf{(10)}(2010), 2727--2756.


    \bibitem{KZH1} K. Zhu, Operator theory on function spaces, Second Edition, Math. Surveys and
Monographs, Vol. 138, American Mathematical Society: Providence, Rhode Island,
2007.

\end{thebibliography}
\end{document}